\documentclass[12pt,leqno]{article}

\usepackage{latexsym,amsmath,amssymb,amsthm,amsfonts}

\renewcommand{\epsilon}{\varepsilon}
\newtheorem{theorem}{Theorem}
\newtheorem{lemma}{Lemma}
\newtheorem{proposition}{Proposition}

\begin{document}

\author{Andriy V. Bondarenko, Danylo V. Radchenko}
\title{On a family of strongly regular graphs with $\lambda=1$}
\date{}
\maketitle
\begin{abstract}
In this paper, we give a complete description of strongly regular
graphs with parameters $((n^2+3n-1)^2,n^2(n+3),1,n(n+1))$. All
possible such graphs are: the lattice graph $L_{3,3}$ with
parameters $(9,4,1,2)$, the Brouwer-Haemers graph with parameters
$(81,20,1,6)$, and the Games graph with parameters $(729,112,1,20)$.
\end{abstract}
{\bf Keywords:} strongly regular graph, automorphism group, Brouwer-Haemers graph, Games graph\\
{\bf AMS subject classification.} 05C25, 05C50, 52C99, 41A55, 11D61

\section{Introduction}

A strongly regular graph $\Gamma$ with parameters
$(v,k,\lambda,\mu)$ is an undirected regular graph on $v$ vertices
of valency $k$ such that each pair of adjacent vertices has
$\lambda$ common neighbors, and each pair of nonadjacent vertices
has $\mu$ common neighbors. The incidence matrix $A$ of $\Gamma$ has
the following properties:
$$
AJ = kJ,
$$
and
$$
A^2 + (\mu - \lambda)A + (\mu - k)I = \mu J,
$$
where $I$ is the identity matrix and $J$ is the matrix with all
entries equal to~$1$. These conditions imply that
\begin{equation}
\label{par} (v - k - 1)\mu = k(k - \lambda - 1).
\end{equation}
Moreover, the matrix $A$ has only 3 eigenvalues: $k$ of multiplicity
$1$, one positive eigenvalue
$$
r=\frac 12\left(\lambda-\mu+\sqrt{(\lambda-\mu)^2+4(k-\mu))}\right)
$$
of multiplicity
$$
f=\frac 12
\left(v-1-\frac{2k+(v-1)(\lambda-\mu)}{\sqrt{(\lambda-\mu)^2+4(k-\mu))}}\right),
$$
and one negative eigenvalue
$$
s=\frac 12\left(\lambda-\mu-\sqrt{(\lambda-\mu)^2+4(k-\mu))}\right)
$$
of multiplicity
\begin{equation}
\label{g} g=\frac 12
\left(v-1+\frac{2k+(v-1)(\lambda-\mu)}{\sqrt{(\lambda-\mu)^2+4(k-\mu))}}\right).
\end{equation}
Clearly, both $f$ and $g$ are integers. This together
with~\eqref{par} gives a family of suitable parameters
$(v,k,\lambda,\mu)$ for strongly regular graphs. The list of all
suitable parameters and known existence results for $v\le 1300$
could be found in~\cite{B} (except for the trivial case when
$\Gamma$ is a disjoint union of complete graphs $mK_n$ or its
complement). Strongly regular graphs often appear in different areas
of group theory, geometry, and set theory. Many of strongly regular
graphs have a large automorphism group, which is the main tool to
construct them. For example, some strongly regular graphs could be
naturally obtained from a rank 3 permutation group, see~\cite{G}. On
the other hand, there are strongly regular graphs having trivial
automorphism groups. The smallest such graph has parameters
(25,12,5,6), see~\cite{CEKS}.

Arguably, the most widely known strongly regular graphs are Moore
graphs, with $\lambda=0$ and $\mu=1$. The list of all suitable
parameters for such graphs are: $(5,2,0,1)$, $(10,3,0,1)$,
$(50,7,0,1)$, and $(3250,57,0,1)$. In the first 3 cases, the graph
with mentioned parameters is unique and has an edge transitive
automorphism group. That is: the cycle graph $C_5$, the Petersen
graph, and the Hoffman-Singleton graph. The question whether exists
a strongly regular graph in the last case is a well-known open
problem posed by Hoffman and Singleton~\cite{HS}. However, Higman
proved that the automorphism group of such a graph could not be even
vertex transitive, see, for example,~\cite{C}. Later, Higman's
approach was widely generalized and applied for other graphs, see,
e.g.,~\cite{MS} and~\cite{MN}. The typical result is that, if for a
given parameter set $(v,k,\lambda,\mu)$ a strongly regular graph
exists, then it has a small automorphism group. Full description of
parameters for which a strongly regular graph exists is not likely
to be ever done.

In this paper we will investigate strongly regular graphs with
$\lambda=1$ and negative eigenspaces of dimension $g=k$. One can
deduce directly from~\eqref{par} and~\eqref{g} that such a graph is
either $K_3$ or belongs to the family
$$
((n^2+3n-1)^2,n^2(n+3),1,n(n+1)),
$$
where $n\in{\mathbb N}$ is the positive eigenvalue.

This family includes the lattice graph $L_{3,3}$ with parameters
$(9,4,1,2)$, the Brouwer-Haemers graph with parameters
$(81,20,1,6)$, which is also known to be unique~\cite{BrHa}, and the
Games graph with parameters $(729,112,1,20)$, for which the
uniqueness question was open. We will show that these are the only
graphs in the family.
\begin{theorem}
Suppose that there exists a strongly regular graph with parameters
$((n^2+3n-1)^2,n^2(n+3),1,n(n+1))$. Then $n\in\{1,2,4\}$.
\end{theorem}

The proof consists of two parts. First, in Section~3 we will show
that each graph in the family exhibits certain symmetries (in
particular, its group of automorphisms is vertex-transitive). Then,
in Section~4 we will use different properties of these symmetries to
prove that the set of vertices can be given a vector space structure
over the finite field $F_3$. In particular, the number of vertices
in the graph is a power of $3$.

Finally, the resulting diophantine equation has the only three
mentioned solutions by virtue of~\cite[Theorem B]{BCHMW}. This
equation has appeared during the studying of ternary linear codes
with exactly two nonzero weights and with the minimal weight of the
dual code at least 4. It was actually shown in~\cite{CK} that each
of such codes of dimension $2m$ implies a strongly regular graph
from the family with $v=3^{2m}$ vertices.

The following result completes description of the family.
\begin{theorem}
The strongly regular graph with parameters $(729,112,1,20)$ is
unique up to isomorphism.
\end{theorem}
We will prove Theorems 1 and 2 in Section 5.

In the next section we will explain the Euclidean representation of
strongly regular graphs, which is the main tool to prove that each
graph in the family has a vertex transitive automorphism group.
\section{Euclidean representation}
Let $\Gamma=(V,E)$ be a strongly regular graph with negative
eigenspace of dimension $g$.
 Then there exists a set of vectors $\{x_i:i\in V\}\subset \mathbb{R}^g$ satisfying the following two conditions. First,
$$
\langle x_i,x_j\rangle=
\begin{cases}1, & \mbox{if }i=j, \\
p, & \mbox{if }i\mbox{ and }j\mbox{ are adjacent},\\
q, & \mbox{else},
\end{cases}
$$
where $p,q\in (-1,1)$. Second condition is that the set $\{x_i:i\in
V\}$ forms a spherical 2-design, that is
$$\sum_{i\in V}x_i=0,$$
and
$$
\sum_{i,j\in V}\langle x_i,x_j\rangle^2 =\frac{|V|^2}{g}.
$$
The values of $p$ and $q$ are uniquely determined by these
conditions. For more information on relations between the Euclidean
representation of strongly regular graphs and spherical designs
see~\cite{Cam}. To construct such vectors consider columns $\{y_i:
i\in V\}$ of the matrix $A-fI$ and put $x_i:=z_i/\|z_i\|$, where
$$
z_i=y_i-\frac 1{|V|}\sum_{j\in V}y_j, \quad i\in V.
$$
It is easy to check that these vectors satisfy the above mentioned
conditions. The main tool we will use for description of strongly
regular graphs is the fact that each subset $\{x_i: i\in U\}$, where
$U\subset V$, has a positive definite Gram matrix
$\left\{(x_i,x_j)\right\}_{i,j\in U}$ of rank at most $g$.
Similarly, we could get the Euclidean representation of $\Gamma$ in
$\mathbb{R}^f$. However, we will never use it in this paper. The
reason is that $f$ is much larger than $g$ for graphs from the
considered family. Thus, the Euclidian representation in
$\mathbb{R}^g$ contains much more information.
\section{Nontrivial automorphisms}
Fix $n\ge 2$. Let $\Gamma=(V,E)$ be a strongly regular graph with
parameters $((n^2+3n-1)^2,n^2(n+3),1,n(n+1))$. In this section, we
will also fix some vertex $v_{\infty}$ of $\Gamma$. For any vertex
$v$, let $N(v)$ be the set of all neighbors of $v$, and let $N'(v)$
be the set of non-neighbors of $v$, i.e. $N'(v)=V\setminus(\{v\}\cup
N(v))$. The subgraph induced on the set $N(v_{\infty})$ is
isomorphic to $m K_2$. Define the permutation $\sigma$ on the set
$\{v_{\infty}\}\cup N(v_{\infty})$ by switching all pairs of
adjacent vertices in $N(v_{\infty})$ and leaving $v_{\infty}$ fixed.
Our goal in this section is to prove the following
\begin{proposition}
The permutation $\sigma$ can be extended uniquely to an automorphism
of $\Gamma$.
\end{proposition}

To prove Proposition 1 we need the next simple lemma.
\begin{lemma}
\label{lem1} Suppose that $m_0,m_1,\ldots,m_{n-2}$ are nonnegative
integers such that
\begin{equation}
\label{l1} \sum_{i=0}^{n-2}\binom{n-i+1}{2}m_i\leq \binom{n+1}{2}.
\end{equation}
Then the following inequality
$$
\sum_{i=0}^{n-2}\binom{n-i}{2}m_i\leq \binom{n}{2},$$ holds. The
equality attains if and only if $m_0=1$, and
$m_1=m_2=\ldots=m_{n-2}=0$.
\end{lemma}
\begin{proof} Since, for each $k<n$,
$$
\binom{k}{2}< \binom{k+1}{2}\frac{\binom{n}{2}}{\binom{n+1}{2}},
$$
we obtain the statement of lemma by multiplying~\eqref{l1} by
$\binom{n}{2}/\binom{n+1}{2}$.
\end{proof}

\begin{proof}[Proof of Proposition 1] For the Euclidean representation of the graph $\Gamma$ in
$\mathbb{R}^g$ we have $p=-\frac{n^2+2n-1}{n^2(n+3)}$, and
$q=\frac{1}{n(n+3)}$. Now computing the determinant of the Gram
matrix we get that vectors $\{x_i:i\in N(v_{\infty})\}$ are linearly
independent. Since $|N(v_{\infty})|=n^2(n+3)=g$, this then implies
that the set $\{x_i:i\in N(v_{\infty})\}$ forms a basis in
$\mathbb{R}^g$.

For any vertex $u\in N'(v_{\infty})$, we have $|N(u)\cap
N(v_{\infty})|=\mu=n(n+1)$. Denote by $A(u)$ the set $N(u)\cap
N(v_{\infty})$,
 and by $B(u)$ the set of neighbors of $A(u)$ that are in $N(v_{\infty}$)
 (so that $|B(u)|=|A(u)|=n(n+1)$). Then, for $\alpha=\frac{n}{n^2-1}$, $\beta=\frac{1}{n^2-1}$ and $\gamma=\frac{n}{n-1}$, we have
\begin{equation}
x=x_u+\alpha\sum_{i\in A(u)}x_i+\beta\sum_{i\in B(u)}x_i+\gamma
x_{v_{\infty}}=0. \label{eq:n1}
\end{equation}
Indeed, it is easy to check that $\langle x,x_i\rangle=0$ for $i\in
N(v_{\infty})$, and since the set $\{x_i:i\in N(v_{\infty})\}$ forms
a basis then $x=0$. Applying equation (\ref{eq:n1}) for each pair
$u,w\in N'(v_{\infty})$ we obtain
\begin{equation}
\begin{cases}
\langle x_u,x_w\rangle = p ~~\Leftrightarrow~~ n|A(u)\cap A(w)|+|A(u)\cap B(w)|=n+1,\\
\langle x_u,x_w\rangle = q ~~\Leftrightarrow~~ n|A(u)\cap A(w)|+|A(u)\cap B(w)|=n(n+1).\\
\end{cases}
\label{eq:n2}
\end{equation}

Observe that for a pair of nonnegative integers $(k,l)$ the equation
$nk+l=n+1$ has only two solutions: $(1,1)$ and $(0,n+1)$. The
equation $nk+l=n(n+1)$ has $n+1$ solutions:
$(n+1,0),(n,n),(n-1,2n),\ldots,(0,n(n+1))$. We see that $|A(u)\cap
A(w)|\leq n+1$, so in any case $A(u)\neq A(w)$. Thus, if we extend
$\sigma$ to an automorphism of $\Gamma$, then $A(\sigma(u))=B(u)$,
and therefore $\sigma(u)$ is defined uniquely.

Now, we fix $u\in N'(v_{\infty})$. Clearly, there are exactly
$n(n+1)$ vertices in $N'(v_{\infty})\cap N(u)$ with $|A(u)\cap
A(w)|=1$, and exactly $n^2(n+3)-2n(n+1)$ such vertices with
$|A(u)\cap A(w)|=0$. Let $m_i$, $i=0,1,\ldots,n+1,$ be the number
 of vertices $w\in N'(v_\infty)\cap N'(u)$ with $|A(u)\cap A(w)|=i$. The numbers $m_i$ must satisfy three equations
\begin{equation}
\label{eq5}
\begin{cases}
\sum_{i=0}^{n+1}m_i=|N'(v_{\infty})\cap N'(u)|=n^4+4n^3+2n^2-5n-1,\\
\sum_{i=0}^{n+1}im_i=|A(u)|(n^2(n+3)-4)=n(n^2-1)(n+2)^2,\\
\sum_{i=0}^{n+1}\binom{i}{2}m_i=n(n^2-1)(n+2)(n^2+n-1)/2.
\end{cases}
\end{equation}

Here, the second equation is obtained by counting edges between
$A(u)$ and $N'(v_{\infty})\cap N'(u)$. We get the third equation by
counting triples $(v_1,v_2,v_3)$, where $v_1$ and $v_2$ are
different vertices in $A(u)$, and $v_3\in N'(v_{\infty})\cap N'(u)$
is adjacent to both $v_1$ and $v_2$.

Now, solving~\eqref{eq5} in $m_{n-1},m_n$ and $m_{n+1}$ we obtain
\begin{equation}
\begin{cases}
m_{n-1}=\binom{n+1}{2}-\sum_{i=0}^{n-2}\binom{n-i+1}{2}m_i,\\
m_{n+1}=n^2(n+3)-2n(n+1)+\binom{n}{2}-\sum_{i=0}^{n-2}\binom{n-i}{2}m_i.\\
\end{cases}
\label{eq:n4}
\end{equation}

The crucial part of the proof is the following inequality
\begin{equation}
\label{main} m_{n+1}\leq n^2(n+3)-2n(n+1)
\end{equation}
to be proven later. Assuming~\eqref{main} it is easy to conclude the
proof of Proposition 1. Indeed, combining~\eqref{eq:n4}
and~\eqref{main} with the statement of Lemma~\ref{lem1} we
immediately obtain that $m_0=1$, and $m_1=m_2=\ldots=m_{n-1}=0$.
Therefore, for each $u\in N'(v_{\infty})$ there exists a unique
vertex $w_u\in N'(v_{\infty})$ such that $A(u)=B(w_u)$. Define
$\sigma$ by $\sigma(u)=w_u$, $u\in N'(v_{\infty})$. Then $\sigma$ is
a bijection, because $\sigma(\sigma(u))=u$. Finally, since for each
pair $u$, $w\in N'(v_{\infty})$
$$|A(u)\cap A(w)|=|B(u)\cap B(w)|=|A(\sigma(u))\cap A(\sigma(w))|,$$
and
$$|A(u)\cap B(w)|=|B(u)\cap A(w)|=|A(\sigma(u))\cap B(\sigma(w))|,$$
we see that $\sigma$ is an automorphism by virtue of (\ref{eq:n2}),
proving Proposition 1.

Now, we are ready to prove~\eqref{main}.

Denote by $C(u)$ the set $\{w\in N'(v_{\infty})\cap N(u):|A(u)\cap
A(w)|=1\}$, and by $D(u)$ the set
 $\{w\in N'(v_{\infty})\cap N'(u) : |A(u)\cap A(w)|=n+1\}$. We have $|C(u)|=|B(u)|=n(n+1)$, and also $|D(u)|=m_{n+1}$.
Thus, to prove~\eqref{main} it is sufficient to show that vectors
$\{x_i:i\in B(u)\cup C(u)\cup D(u)\}$ are linearly independent
(recall that $g=n^2(n+3)$).

First, note that there are no edges between $B(u)\cup C(u)$ and
$D(u)$, between $v_{\infty}$ and $C(u)$, and between $u$ and $B(u)$.
Then, $v_{\infty}$ is connected to each vertex in $B(u)$, and $u$ is
connected to each vertex in $C(u)$. Also, in the subgraph induced on
$B(u)\cup C(u)$ every vertex has a degree $1$. Suppose that there is
a nontrivial linear relation
\begin{equation}
\sum_{i\in B(u)}\beta_ix_i+\sum_{i\in C(u)}\gamma_ix_i+\sum_{i\in
D(u)}\delta_ix_i=0. \label{eq:n5}
\end{equation}
For $i\in B(u)$, denote by $\phi(i)\in V$ its unique neighbor in
$C(u)$. Put $S_B=\sum_{i\in B(u)}\beta_i$, $S_C=\sum_{i\in
C(u)}\gamma_i$, and $S_D=\sum_{i\in D(u)}\delta_i$. Now, taking
inner products of both sides in~(\ref{eq:n5}) with $x_{v_{\infty}}$
and with $x_u$, we get
\begin{equation}
\begin{cases}
pS_B+qS_C+qS_D=0,\\
qS_B+pS_C+qS_D=0.
\end{cases}
\label{eq:n6}
\end{equation}

Similarly, taking inner products of both sides in~(\ref{eq:n5}) with
$x_i$ and with $x_{\phi(i)}$ for each $i\in B(u)$, we obtain
\begin{equation}
\begin{cases}
(1-q)\beta_i+(p-q)\gamma_{\phi(i)}+qS_B+qS_C+qS_D=0,\\
(p-q)\beta_i+(1-q)\gamma_{\phi(i)}+qS_B+qS_C+qS_D=0.
\end{cases}
\label{eq:n7}
\end{equation}
Subtracting from the first equation in (\ref{eq:n6}) the second we
get that $S_B=S_C$. Similarly,~\eqref{eq:n7} yields
$\beta_i=\gamma_{\phi(i)}$, $i\in B(u)$. Now, summing up the first
equation of (\ref{eq:n7}) over all $i\in B(u)$ we obtain
\begin{equation*}
(1+p+2(n^2+n-1)q)S_B+n(n+1)qS_D=0.
\end{equation*}
Combining this equation with (\ref{eq:n6}) we get $S_B=S_C=S_D=0$.
Thus,~\eqref{eq:n7} together with inequality $1+p-2q>0$, $n\ge 2,$
implies that $\beta_i=\gamma_{\phi(i)}=0$ for all $i\in B(u)$.
Therefore, we are left with the relation
$$
\sum_{i\in D(u)}\delta_ix_i=0,
$$
where $\sum_{i\in D(u)}\delta_i=0$. Let
$\delta_i^{+}=\max(\delta_i,0)$, and
$\delta_i^{-}=-\min(\delta_i,0)$. Then we have
$$
\sum_{i\in D(u)}\delta_i^{+}x_i=\sum_{i\in D(u)}\delta_i^{-}x_i.
$$
Normalizing $\delta_i$ in such a way that $\sum_{i\in
D(u)}\delta_i^{+}=\sum_{i\in D(u)}\delta_i^{-}=1$ we obtain the
following bound
\begin{align}
\label{eq12} \|\sum_{i\in D(u)}\delta_i^{+}x_i\|^2&=\langle
\sum_{i\in D(u)}\delta_i^{+}x_i,\sum_{i\in
D(u)}\delta_i^{-}x_i\rangle
=\sum_{i,j\in D(u)}\delta_i^{+}\delta_j^{-}\langle x_i,x_j\rangle\\
\notag &=\sum_{i,j\in D(u), i\neq j}\delta_i^{+}\delta_j^{-}\langle
x_i,x_j\rangle \leq q\sum_{i,j\in D(u), i \neq
j}\delta_i^{+}\delta_j^{-}=q,
\end{align}
where we used the fact that $\delta_i^{+}\delta_i^{-}=0$ for $i\in
D(u)$, and $\langle x_i,x_j\rangle\leq q$ for $i\neq j$. We will now
use~\eqref{eq12} to show that the quadratic form of three variables
$$
Q(a,b,c)=\|a(x_u+x_{v_{\infty}})+b\sum_{i\in B(u)\cup
C(u)}x_i+c\sum_{i\in D(u)}\delta_i^{+}x_i\|^2
$$
attains a negative value for some $a,b,c\in{\mathbb R}$, thus
contradicting our assumption (\ref{eq:n5}). Let $(a_{ij})$ be the
symmetric $3\times 3$ matrix associated to the quadratic form $Q$.
Its entries are:
\begin{align*}
a_{11}&=\|x_u+x_{v_{\infty}}\|^2=2+2q,\\
a_{12}&=(|B(u)|p+|C(u)|q)+(|B(u)|q+|C(u)|p)=2n(n+1)(p+q),\\
a_{13}&=2q\sum_{i\in D(u)}\delta_i^+=2q,\\
a_{22}&=|B(u)\cup C(u)|(1+p+2(n^2+n-1)q)=2n(n+1)(1+p+2(n^2+n-1)q),\\
a_{23}&=|B(u)\cup C(u)|\sum_{i\in D(u)}\delta_i^+q=2n(n+1)q,\\
a_{33}&=\|\sum_{i\in D(u)}\delta_i^{+}x_i\|^2\leq q.
\end{align*}
We may assume that $a_{33}=q$, since $Q(a,b,c)$ can only increase.
The determinant of the resulting matrix is
$$
\det(a_{ij})=-\frac{8(n+1)(n^4+6n^3+7n^2-6n+1)}{n^3(n+3)^3},
$$
which is negative for $n\geq 1$. Thus, $Q$ attains a negative value.
Therefore vectors $\{x_i:i\in B(u)\cup C(u)\cup D(u)\}$ are linearly
independent, proving~\eqref{main} and hence Proposition 1.
\end{proof}
For each $v\in V$, denote by $\sigma_{v}$ the automorphism
constructed for $v_{\infty}=v$. It follows from the definition that
each $\sigma_v$ is an involution. Moreover, for each three vertices
$u$, $v$, and $w$ forming a triangle, $\sigma_u(v)=w$.  Since
$\Gamma$ is a connected graph this immediately implies that
$Aut(\Gamma)$ is vertex transitive.

\section{Structure of automorphisms}

We now proceed to study the global structure of strongly regular
graphs with parameters $$((n^2+3n-1)^2,n^2(n+3),1,n(n+1))$$ by using
the automorphisms we have constructed in Proposition 1.  In this
section, we will prove the following
\begin{proposition}
The set of vertices $V$ can be given a vector space structure over
$F_3$ such that $\sigma_v(u)=-(u+v)$ for all $u,v\in V$, in
particular, $|V|=3^m$ for some $m\in \mathbb{N}$.
\end{proposition}

First, we establish some properties of $\sigma_v$.
\begin{lemma}
Involutions $\{\sigma_u: u\in V\}$ satisfy the following:
\begin{itemize}
\item[(i)] $\sigma_u(v)=\sigma_v(u)$ for all $u,v\in V$;
\item[(ii)] if $g$ is any automorphism of $\Gamma$, then $g\sigma_u=\sigma_{g(u)}g$;
\item[(iii)] $\sigma_u\sigma_v\sigma_u=\sigma_{\sigma_u(v)}$ for all $u,v\in V$;
\item[(iv)]  the automorphism $\sigma_u\sigma_v$ has no fixed points, and $(\sigma_u\sigma_v)^3=e$ for all $u\neq v\in V$;
\item[(v)] if $u$ and $v$ are adjacent, then for all $x\in V$ the vertices $x$ and $\sigma_u\sigma_v(x)$ are also adjacent.
\end{itemize}
\end{lemma}
\begin{proof}

If $u$ and $v$ are adjacent, then they have both vertices
$\sigma_u(v)$ and $\sigma_v(u)$ as common neighbors. Thus the
condition $\lambda=1$ implies that $\sigma_u(v)=\sigma_v(u)$. If $u$
and $v$ are nonadjacent, then $\sigma_u(v)$ is a unique vertex $w$,
such that $w$ is nonadjacent to both $u$ and $v$, and $N(u)\cap
N(v)\cap N(w)=\emptyset$. Hence, in this case we also have
$\sigma_u(v)=\sigma_v(u)$. This proves (i).

From (i) we see that triples $\{u,v,w\}$ with $w=\sigma_u(v)$ are
defined symmetrically in $u,v,w$, that is the following identities
$$
w=\sigma_u(v)=\sigma_v(u),\quad v=\sigma_u(w)=\sigma_w(u),\quad
u=\sigma_w(v)=\sigma_v(w)
$$
hold. Therefore, $g$ must preserve the set of such triples. This
means that $\sigma_{g(u)}(g(v))=g(w)=g(\sigma_u(v))$, which proves
(ii).

Part (iii) follows from (ii) by letting $g=\sigma_v$.

From (iii) we have
$\sigma_u\sigma_v\sigma_u=\sigma_{\sigma_u(v)}=\sigma_{\sigma_v(u)}=\sigma_v\sigma_u\sigma_v$,
hence $(\sigma_u\sigma_v)^3=e$. If $\sigma_u(\sigma_v(x))=x$, then
$\sigma_u(x)=\sigma_v(x)$, and therefore $\sigma_x(u)=\sigma_x(v)$,
which contradicts $u\neq v$. This proves (iv).

Finally, since $\sigma_x$ is a graph automorphism for each $x\in V$,
then for an adjacent pair of vertices $u$ and $v$ the vertices
$\sigma_x(u)$ and $\sigma_x(v)$ are also adjacent. Similarly, the
vertices $\sigma_u(\sigma_x(u))=x$ and
$\sigma_u(\sigma_x(v))=\sigma_u\sigma_v(x)$ are adjacent as well.
This proves (v).
\end{proof}

In the sequel, we will use the notation $u\circ v$ for
$\sigma_u(v)$. It is convenient, since for any automorphism $g$ we
have $g(u\circ v)=g(u)\circ g(v)$. The crucial part in the proof of
Proposition 2 is the following lemma.
\begin{lemma}
For all $u$, $v$, $w\in V$, we have
$(\sigma_v\sigma_w\sigma_u)^2=e$.
\end{lemma}
\begin{proof}
We may assume that $u,v,w$ are different vertices, as other cases
are covered by Lemma 2.


Denote $g_1=\sigma_u\sigma_v$, $g_2=\sigma_{u\circ
w}\sigma_{w\circ(u\circ v)}$, and $g_3=\sigma_w\sigma_{u\circ(v\circ
w)}$. Since $u\neq v$ then $u\circ w\neq w\circ(u\circ v)$, and
$w\neq u\circ(v\circ w)$. Hence, by Lemma 2 (iv), the automorphisms
$g_1$, $g_2$, and $g_3$ are of order $3$, and have no fixed points.
Moreover, we claim that $g_1g_2g_3=e$. Indeed, we have
\begin{gather*}
g_1=\sigma_u\sigma_v,\\
g_2=\sigma_{u\circ w}\sigma_{w\circ(u\circ v)}=\sigma_w(\sigma_u\sigma_w\sigma_w\sigma_u)\sigma_v\sigma_u\sigma_w=\sigma_w\sigma_v\sigma_u\sigma_w,\\
g_3=\sigma_w\sigma_{u\circ(v\circ
w)}=\sigma_w\sigma_u\sigma_w\sigma_v\sigma_w\sigma_u,
\end{gather*}
where we used $\sigma_{x\circ
y}=\sigma_{\sigma_x(y)}=\sigma_x\sigma_y\sigma_x$ repeatedly.
Therefore,
\begin{gather*}
g_1g_2g_3=\sigma_u\sigma_v\sigma_w\sigma_v(\sigma_u\sigma_w\sigma_w\sigma_u)\sigma_w\sigma_v\sigma_w\sigma_u=\\
\sigma_u\sigma_v\sigma_w\sigma_v\sigma_w\sigma_v\sigma_w\sigma_u=\sigma_u(\sigma_v\sigma_w)^3\sigma_u=\sigma_u^2=e.
\end{gather*}

Let us first suppose that $u$ and $v$ are adjacent. Then, by Lemma 2
(v), the vertices $u\circ w$ and $w\circ (u\circ v)$ are also
adjacent. Similarly, $w$ and $u\circ (v\circ w)$ are adjacent. Using
again Lemma 2 (v), we see that $\{x,g_i(x),g_i^{-1}(x)\}$ is a
triangle for any $x\in V$ and $i\in\{1,2,3\}$. From $g_1g_2g_3=e$ we
find that $g_3(x)=g_2^{-1}(g_1^{-1}(x))$, therefore $g_1^{-1}(x)$
and $g_3(x)$ are adjacent. Hence $x$ and $g_1^{-1}(x)$ are also
adjacent, and have two common neighbors $g_1(x)$ and $g_3(x)$. Since
$\lambda=1$ then $g_1(x)=g_3(x)$. The choice of $x$ was arbitrary,
therefore $g_1=g_3$. Thus we obtain
\begin{gather*}
g_1=g_3 \Leftrightarrow \sigma_u\sigma_v=\sigma_w\sigma_u\sigma_w\sigma_v\sigma_w\sigma_u \Leftrightarrow\\
\sigma_v(\sigma_u\sigma_w)^2\sigma_v\sigma_w\sigma_u=e
\Leftrightarrow (\sigma_v\sigma_w\sigma_u)^2=e,
\end{gather*}
as claimed.

Now, suppose that $u$ and $v$ are nonadjacent. Then there exists
$t\in V$ adjacent to both $u$ and $v$. By the previous case, we have
\begin{gather*}
(\sigma_t\sigma_w\sigma_u)^2=e,\\
(\sigma_v\sigma_w\sigma_t)^2=e,\\
\sigma_t\sigma_v\sigma_u\sigma_t=\sigma_u\sigma_v,
\end{gather*}
where the last equation is another form of
$(\sigma_t\sigma_v\sigma_u)^2=e$. Multiplying the first two
equations together, we get
\begin{gather*}
e=(\sigma_v\sigma_w\sigma_t)^2(\sigma_t\sigma_w\sigma_u)^2=\\
\sigma_v\sigma_w\sigma_t\sigma_v(\sigma_w\sigma_t\sigma_t\sigma_w)\sigma_u\sigma_t\sigma_w\sigma_u=\\
\sigma_v\sigma_w(\sigma_t\sigma_v\sigma_u\sigma_t)\sigma_w\sigma_u=\\
\sigma_v\sigma_w\sigma_u\sigma_v\sigma_w\sigma_u=(\sigma_v\sigma_w\sigma_u)^2.
\end{gather*}
Lemma 3 is proved.
\end{proof}

\begin{proof}[Proof of Proposition 2]
Now we can introduce a vector space structure on $V$. Take some
vertex $v_0\in V$ and define it to be the zero vector. Denote the
multiplication of a vector $v\in V$ by scalars from $F_3$ as follows
$$
1v=v, \quad 0v=v_0, \quad \text{and}\quad (-1)v=\sigma_{v_0}(v).
$$
Finally, for each $u,v\in V$ define vector addition by
$$u+v=\sigma_{v_0}(u\circ v).$$ We only need to check that these
operations indeed define a vector space over $F_3$. First, let us
show the associativity of the vector addition. For arbitrary $u$,
$v$, $w\in V$, we have
\begin{gather*}
u+(v+w)=(u+v)+w \Leftrightarrow \sigma_{v_0}(u\circ(v+w))=\sigma_{v_0}((u+v)\circ w) \Leftrightarrow \\
u\circ(v+w)=(u+v)\circ w \Leftrightarrow \sigma_u\sigma_{v_0}(v\circ w)=\sigma_w\sigma_{v_0}(v\circ u) \Leftrightarrow \\
\sigma_u\sigma_{v_0}\sigma_w(v)=\sigma_w\sigma_{v_0}\sigma_u(v)\Leftrightarrow
(\sigma_u\sigma_{v_0}\sigma_w)^2(v)=v,
\end{gather*}
which is true by Lemma 3. The rest of axioms are easily deduced from
Lemma 2. Namely, for each $u$, $v\in V$ we have
$$
u+v=\sigma_{v_0}(u\circ v)=\sigma_{v_0}(v\circ u)=v+u,
$$
$$
v+v_0=\sigma_{v_0}\sigma_{v_0}(v)=v, \quad
(-1)((-1)v)=\sigma_{v_0}\sigma_{v_0}(v)=v,
$$
$$
v+v=\sigma_{v_0}\sigma_{v}(v)=\sigma_{v_0}v=(-1)v,\quad v+(-1)v=
\sigma_{v_0}\sigma_{v}\sigma_{v_0}(v)=\sigma_{v}\sigma_{v_0}\sigma_{v}(v)=v_0,
$$
and
$$
-(u+v)=\sigma_{v_0}\sigma_{v_0}(u\circ
v)=\sigma_{u}(v)=\sigma_{v_0}(u)+\sigma_{v_0}(v)=(-1)u+(-1)v.
$$
Proposition 2 is proved.
\end{proof}

As a corollary, we see that each strongly regular graph in the
family must have $3^m$ vertices, $m\in{\mathbb N}$. For example,
there is no strongly regular graph with parameters $(289,54,1,12)$.
\section{Proof of main results}
Now we are ready to prove Theorem 1.
\begin{proof}[Proof of Theorem 1]
By Propositions 1 and 2 it remains to show that the diophantine
equation $n^2+3n-1=3^m$ has only the following integer solutions
$(n,m)$ with $n>0$:
$$ (1,1),\quad (2,2),\quad (4,3).$$
Substituting $u=2n+3$ we obtain an equivalent equation
$$
u^2-13=4\cdot3^m.
$$
By~\cite[Theorem B]{BCHMW}, this equation can hold only for $m=1,2$
or $3$.

\end{proof}
\noindent With some more work we can prove that there exists a
unique strongly regular graph with parameters $(729,112,1,20)$,
namely the Games graph.

\begin{proof}[Proof of Theorem 2] Let $\Gamma=(V,E)$ be a strongly regular
graph with parameters $(729,112,1,20)$. By Proposition 2 we may
assume that $V=F_3^6$, and for any $x\in V$ the map
$\sigma_x:t\rightarrow -(t+x)$ is a graph automorphism. Since
$\sigma_0\sigma_x(t)=t+x$, we have that all shifts $t\rightarrow
x+t$ are also automorphisms. Therefore $x$ and $y$ are adjacent if
and only if $0$ and $x-y$ are adjacent. The last implies that $0$
and $x$ are adjacent if and only if $0$ and $-x$ are adjacent.
Hence, we can form a set $H\subset PG(5,F_3)$ by saying $[x]\in H$
if and only if $0$ and $x$ are adjacent ($[x]$ is the equivalence
class of $x\in F_3^6\setminus\{0\}$ in $PG(5,F_3)$). The set $H$
consists of $56$ points. Now, we claim that $H$ is a cap, i.e., any
line in $PG(5,F_3)$ meets $H$ in at most two points. Indeed, it is
sufficient to consider three collinear points $[x],[y],[x+y]\in H$.
Then $0$ is connected to $x$, $y$, and $x+y$. Similarly, $x+y$ is
connected to $x$, $y$, and $0$. This contradicts the fact that
$\lambda=1$, as $0$ and $x+y$ are connected and have $x$ and $y$ as
common neighbors. This proves that $H$ is a cap. It was shown by
Hill \cite{Hill} that there exists a unique cap with $56$ points in
$PG(5,F_3)$, and therefore $H$ is isomorphic to it. Hence $\Gamma$
is isomorphic to the Games graph, as can be seen from the
construction given in \cite[p.114-115]{BrvL}.

\end{proof}

{\footnotesize \noindent Centre de Recerca Matem\`atica, Campus de
Bellaterra, 08193
Bellaterra (Barcelona), Spain\\
and\\
Department of Mathematical Analysis, National Taras Shevchenko
University, str.\ Volodymyrska, 64, Kyiv, 01033, Ukraine\\
{\it Email address: andriybond@gmail.com}\\
\vspace{0.5cm}

\noindent Department of Mathematical Analysis, National Taras
Shevchenko University, str.\ Volodymyrska, 64, Kyiv, 01033, Ukraine\\
{\it Email address: danradchenko@gmail.com}\\

\end{document}